\documentclass[twoside]{article}
\usepackage{amssymb,enumerate}
\usepackage{amsmath,amscd}

\def\part#1{\frac{\partial\phantom{#1}}{\partial#1}}
\newtheorem{thm}{Theorem}
\newtheorem{theorem}[thm]{Theorem}
\newtheorem{proposition}[thm]{Proposition}
\newtheorem{lemma}[thm]{Lemma}

\newtheorem{conjecture}[thm]{Conjecture}

\newtheorem{thesis}[thm]{Thesis}
\newenvironment{proof}{\begin{trivlist}\item[]{\bf Proof} }%
{\hfill $\Box$ \end{trivlist}}
\newenvironment{definition}{\begin{trivlist}\item[]{\bf Definition}\em }%
{\end{trivlist}}
\newenvironment{remark}{\begin{trivlist}\item[]{\bf Remark} }%
{\end{trivlist}}
\newenvironment{example}{\begin{trivlist}\item[]{\bf Example} }%
{\end{trivlist}}
{\end{trivlist}}

\def\Z{\ifmmode{{\mathbb Z}}\else{${\mathbb Z}$}\fi}
\def\Q{\ifmmode{{\mathbb Q}}\else{${\mathbb Q}$}\fi}
\def\C{\ifmmode{{\mathbb C}}\else{${\mathbb C}$}\fi}
\def\P{\ifmmode{{\mathbb P}}\else{${\mathbb P}$}\fi}
\def\R{\ifmmode{{\mathbb R}}\else{${\mathbb R}$}\fi}

\def\H{\ifmmode{{\mathrm H}}\else{${\mathrm H}$}\fi}
\def\N{\ifmmode{{\mathrm N}}\else{${\mathrm N}$}\fi}

\def\A{\ifmmode{{\mathcal A}}\else{${\mathcal A}$}\fi}
\def\B{\ifmmode{{\mathcal B}}\else{${\mathcal B}$}\fi}
\def\E{\ifmmode{{\mathcal E}}\else{${\mathcal E}$}\fi}
\def\F{\ifmmode{{\mathcal F}}\else{${\mathcal F}$}\fi}
\def\G{\ifmmode{{\mathcal G}}\else{${\mathcal G}$}\fi}
\def\I{\ifmmode{{\mathcal I}}\else{${\mathcal I}$}\fi}
\def\K{\ifmmode{{\mathcal K}}\else{${\mathcal K}$}\fi}
\def\L{\ifmmode{{\mathcal L}}\else{${\mathcal L}$}\fi}
\def\M{\ifmmode{{\mathcal M}}\else{${\mathcal M}$}\fi}
\def\O{\ifmmode{{\mathcal O}}\else{${\mathcal O}$}\fi}
\def\U{\ifmmode{{\mathcal U}}\else{${\mathcal U}$}\fi}
\def\V{\ifmmode{{\mathcal V}}\else{${\mathcal V}$}\fi}
\def\X{\ifmmode{{\mathcal X}}\else{${\mathcal X}$}\fi}

\def\Br{\ifmmode{{\mathrm{Br}}}\else{${\mathrm{Br}}$}\fi}
\def\OG{\ifmmode{\widetilde{\cal M}_4}\else{$\widetilde{\cal M}_4$}\fi}
\def\D{\ifmmode{{\mathcal D}^b}\else{${{\mathcal
    D}^b}$}\fi}
\def\Shah{\ifmmode{\amalg\hspace*{-3.5pt}\amalg}\else{$\amalg\hspace*{-3.5pt}\amalg$}\fi}

\begin{document}

\title{Moduli spaces of sheaves on K3 surfaces}
\author{Justin Sawon}
\date{February, 2016}
\maketitle

\begin{abstract}
In this survey article we describe moduli spaces of simple, stable, and semistable sheaves on K3 surfaces, following the work of Mukai, O'Grady, Huybrechts, Yoshioka, and others. We also describe some recent developments, including applications to the study of Chow rings of K3 surfaces, determination of the ample and nef cones of irreducible holomorphic symplectic manifolds, and moduli spaces of Bridgeland stable complexes of sheaves.
\end{abstract}

\maketitle

\section{Introduction}

In the 60s, moduli spaces of vector bundles on curves were much studied, following Mumford's general construction using Geometric Invariant Theory. Attention then turned to vector bundles on surfaces and higher dimensional varieties. Several new features immediately arose. A vector bundle on a curve has two numerical invariants, the rank $r$ and degree $d$ (equivalently, the first Chern class $c_1$), and the construction of moduli spaces as GIT quotients leads naturally to the definition of stability in terms of the slope $\mu=d/r$. Slope stability extends to higher dimensions, but by using the additional numerical invariants (higher Chern classes $c_2$, etc.) Gieseker was led to another notion of stability. While slope stability is easier to work with when manipulating vector bundles, Gieseker stability allows for a more general construction of moduli spaces; of course, these two notions of stability are closely related, and in many cases equivalent. Another difference in higher dimensions is that, unlike on smooth curves, torsion free coherent sheaves are not automatically locally free. Thus to obtain compact moduli spaces, one must consider not just vector bundles but also torsion free sheaves.

Moduli spaces of sheaves on $\P^2$ and on higher dimensional projective spaces were very explicitly described in the late 70s, partly motivated by the relation to Yang-Mills instantons on $\R^4$ and $S^4$. The geometry of moduli spaces of sheaves on other specific surfaces were also studied, notably for K3 and abelian surfaces by Mukai in the 1980s. Such moduli spaces reflect the geometry of the underlying surface; they can be used to reveal properties of the surface, or they may be studied as interesting spaces in their own right. For instance, Mukai proved that moduli spaces of sheaves on K3 surfaces are examples of holomorphic symplectic manifolds, i.e., higher dimensional analogues of K3 surfaces, or from a differential geometric perspective, hyperk{\"a}hler manifolds. These moduli spaces have been intensely studied ever since, both for what they tell us about K3 surfaces (and their curves, linear systems, Chow groups, etc.) and as test cases for uncovering the structure of general holomorphic symplectic manifolds. Another impetus for studying moduli spaces of sheaves on surfaces, from the mid-80s to the mid-90s, was Donaldson's fundamental work on the topology of smooth four-manifolds via moduli spaces of Yang-Mills instantons. Because of the Hitchin-Kobayashi correspondence, which gives a Hermitian-Einstein connection on any stable bundle, Donaldson's polynomial invariants for algebraic surfaces can be computed using moduli spaces of stable bundles.

In this article we survey moduli spaces of sheaves on K3 surfaces. In Section~2 we recall simpleness, slope stability, and Gieseker stability for torsion free coherent sheaves, and describe how these notions are related. We describe some general properties of moduli spaces of simple, stable, and semistable sheaves on K3 surfaces. We end the section with a short discussion of the Hitchin-Kobayashi correspondence. Much of the material in this section remains valid for arbitrary projective surfaces, and almost every result for K3 surfaces has an analogous result for abelian surfaces. The rest of the article describes some recent developments and applications of sheaves on K3 surfaces, and their moduli spaces. In Section~3 we review work of Beauville, Voisin, Huybrechts, and O'Grady on the Chow ring of a K3 surface. In Section~4 we discuss Hassett and Tschinkel's conjectural description of the Mori cone of curves on a holomorphic symplectic manifold, which determines the ample and nef cones, and what is known for moduli spaces of sheaves on K3 surfaces. In Section~5 we give a brief description of Bridgeland stability conditions on derived categories, and present some applications of these ideas to moduli spaces of sheaves and complexes of sheaves on K3 surfaces due to Bayer and Macr{\`i}.

For the reader interested in pursuing these topics further, we recommend the excellent survey articles by Mukai~\cite{mukai88} and Yoshioka~\cite{yoshioka07}, which, like this article, concentrate mainly on moduli spaces of sheaves on K3 surfaces. A more general introduction to moduli spaces of sheaves on algebraic surfaces can be found in the books of Friedman~\cite{friedman98}, Huybrechts and Lehn~\cite{hl97}, and Le Potier~\cite{lepotier97}.

This article is based on a talk by the author at the ``Workshop on Instanton Counting: Moduli Spaces, Representation Theory and Integrable Systems'' (Leiden, 16--20 June 2014). The author would like to thank the organizers (Ugo Bruzzo, Dimitri Markushevich, Vladimir Rubtsov, Francesco Sala, and Sergey Shadrin) for the invitation, and the Lorentz Center for hospitality. Conversations with Aaron Bertram, Emanuele Macr{\`i}, and Kieran O'Grady were helpful in shaping this article, as were the comments of an anonymous referee. The author gratefully acknowledges support from the NSF, grant number DMS-1206309.

\section{Simple and stable sheaves, and their moduli spaces}

\subsection{Simple sheaves}

Let $X$ be a K3 surface. When considering families of holomorphic bundles or coherent sheaves on $X$, we fix the underlying topological type of the bundles or sheaves. For this purpose, it is convenient to introduce the following lattice.

\begin{definition}
The Mukai lattice of $X$ is
$$\H^{\mathrm{ev}}(X,\mathbb{Z}):=\H^0(X,\mathbb{Z})\oplus\H^2(X,\mathbb{Z})\oplus\H^4(X,\mathbb{Z})$$
with the pairing defined by
$$\langle v,w\rangle:=\int_X-v_0w_4+v_2w_2-v_4w_0.$$
The Mukai vector of a sheaf $\mathcal{E}$ on $X$ is defined to be
$$v(\mathcal{E}):=\mathrm{ch}(\mathcal{E})\mathrm{Td}^{\frac{1}{2}}_X=\left(r,c_1,\frac{c_1^2}{2}-c_2+r\right),$$
where $c_1$ and $c_2$ are the Chern classes of $\mathcal{E}$, and $r$ its rank. We call $\mathcal{E}$ a simple sheaf if $\mathrm{Hom}(\mathcal{E},\mathcal{E})\cong\mathbb{C}$.
\end{definition}

Under fairly general conditions, moduli of coherent sheaves exists only as Artin stacks. By adding the simpleness condition, we obtain moduli spaces that are algebraic spaces (possibly non-Hausdorff). For simple sheaves on K3 surfaces, Mukai proved more.

%Mukai considered moduli spaces of simple sheaves on K3 surfaces. The reason for considering simple sheaves is that, in general, objects with additional automorphisms lead to singularities of the moduli space.

\begin{thm}[Mukai~\cite{mukai84}]
Fix an element $v=(v_0,v_2,v_4)$ in $\mathrm{H}^{\mathrm{ev}}(X,\mathbb{Z})$ with $v_2$ lying in the N{\'e}ron-Severi group $\mathrm{NS}(X)$ of $X$. The moduli space $\mathrm{Spl}(v)$ of simple sheaves $\mathcal{E}$ on $X$ with $v(\mathcal{E})=v$ is smooth of dimension $\langle v,v\rangle+2$.
Moreover, $\mathrm{Spl}(v)$ admits a closed non-degenerate holomorphic two-form.
\end{thm}

\begin{proof}
We give a rough sketch of the argument and refer to~\cite{mukai84} for details (see also~\cite{mukai87} and~\cite{mukai88}). The Zariski tangent space of $\mathrm{Spl}(v)$ at $\mathcal{E}$ is isomorphic to $\mathrm{Ext}^1(\mathcal{E},\mathcal{E})$. Riemann-Roch gives
\begin{eqnarray*}
\chi(\E,\E) & = & \mathrm{dim}\mathrm{Hom}(\E,\E)-\mathrm{dim}\mathrm{Ext}^1(\E,\E)+\mathrm{dim}\mathrm{Ext}^2(\E,\E) \\
 & = & \int_X \mathrm{ch}(\E^{\vee})\mathrm{ch}(\E)\mathrm{Td}_X \\
 & = & -\langle v,v\rangle.
\end{eqnarray*}
Now $\mathrm{Hom}(\E,\E)\cong\C$ is one-dimensional, and so is $\mathrm{Ext}^2(\E,\E)\cong\mathrm{Hom}(\E,\E)^{\vee}$, by Serre duality and $\omega_X\cong\O_X$. Therefore
$$\mathrm{dim}\mathrm{Ext}^1(\E,\E)=\langle v,v\rangle+2.$$
This gives the dimension, though we still have to establish smoothness. In general, the obstruction to deforming a sheaf $\E$ lies in $\mathrm{Ext}^2(\E,\E)$. One can prove that taking the trace gives the obstruction in $\H^2(\O_X)$ to deforming the line bundle $\det\E$. But there is no obstruction to deforming a line bundle on a K3 surface; indeed, line bundles are rigid. Now for a simple sheaf $\E$, the map
$$\mathrm{Ext}^2(\E,\E)\stackrel{\mathrm{tr}}{\longrightarrow}\H^2(\O_X)$$
is an isomorphism, as it is the dual of
$$\H^0(\O_X)\cong\C\longrightarrow \mathrm{Hom}(\E,\E).$$
Thus the vanishing of the obstruction to deforming $\det\E$ implies the vanishing of the obstruction to deforming $\E$ itself, and $\mathrm{Spl}(v)$ is smooth at $\E$.

A bilinear pairing on the tangent space $\mathrm{Ext}^1(\E,\E)$ is given by composing and then taking the trace:
$$\mathrm{Ext}^1(\E,\E)\times\mathrm{Ext}^1(\E,\E)\longrightarrow\mathrm{Ext}^2(\E,\E)\stackrel{\mathrm{tr}}{\longrightarrow}\H^2(\O_X)\cong\C$$
If $a\in\mathrm{Ext}^1(\E,\E)$ then the pairing of $a$ with itself gives the obstruction to deforming $\E$ in the direction $a$, as described above. Since this obstruction vanishes for simple sheaves $\E$, the pairing is skew-symmetric. For holomorphicity and closedness of the resulting two-form on $\mathrm{Spl}(v)$ see~\cite{mukai88}, page 154.
\end{proof}

\subsection{Stability and semistability}

While the moduli space $\mathrm{Spl}(v)$ has some nice properties, it is usually non-Hausdorff, and for this reason we consider instead stable sheaves. Let $(X,H)$ be a polarized K3 surface, i.e., $H$ is an ample divisor on $X$.

\begin{definition}
The slope of a torsion free coherent sheaf $\E$ is defined to be
$$\mu(\E):=\frac{c_1(\E).H}{\mathrm{rank}\,\E}.$$
A torsion free coherent sheaf $\E$ is $\mu$-stable (respectively $\mu$-semistable) if $\mu(\F)<\mu(\E)$ (respectively $\mu(\F)\leq\mu(\E)$) for all coherent subsheaves $\F\subset\E$ with $0<\mathrm{rank}\,\F<\mathrm{rank}\,\E$.
\end{definition}

\begin{remark}
This is also known as Mumford-Takemoto stability~\cite{takemoto72} and slope stability.
\end{remark}

\begin{definition}
The normalized Hilbert polynomial of a torsion free coherent sheaf $\E$ is defined to be
$$p_{H,\E}(n):=\frac{\chi(\E\otimes H^n)}{\mathrm{rank}\,\E}.$$
A torsion free coherent sheaf $\E$ is stable (respectively semistable) if $p_{H,\F}(n)<p_{H,\E}(n)$ (respectively $p_{H,\F}(n)\leq p_{H,\E}(n)$) for $n\gg 0$ for all proper subsheaves $\F\subset\E$.
\end{definition}

\begin{remark}
This is also known as Gieseker stability~\cite{gieseker77} and Maruyama stability~\cite{maruyama78}. (To emphasize the dependence on the choice of polarization, $H$-stability is also used; but care needs to be taken since $H$-stable sometimes means $\mu$-stable with respect to $H$.) In addition, Simpson~\cite{simpson94} generalized the notion of stability to pure-dimensional sheaves.
\end{remark}

\begin{lemma}
\label{2}
\begin{enumerate}
\item Stable sheaves are simple.
\item We have the following implications:
$$\mbox{$\E$ is $\mu$-stable }\Rightarrow\mbox{ $\E$ is stable }\Rightarrow\mbox{ $\E$ is semistable }\Rightarrow\mbox{ $\E$ is $\mu$-semistable}$$
\item If the rank of $\E$ and divisibility of $c_1(\E)$ in $\H^2(X,\Z)$ are coprime, and $H$ is $v$-general where $v=v(\E)$, then $\E$ is $\mu$-semistable implies $\E$ is $\mu$-stable. (So in this case, all four notions of stability and semistability in 2 are equivalent for $\E$.)
\item If $v(\E)\in\H^{\mathrm{ev}}(X,\Z)$ is primitive and $H$ is $v$-general, then $\E$ is semistable implies $\E$ is stable.
\end{enumerate}
\end{lemma}

\begin{remark}
For each $v\in\H^{\mathrm{ev}}(X,\Z)$ there are hyperplanes in the ample cone of $X$ known as walls. We call an ample divisor $H$ $v$-general if it does not lie on any of these walls.
\end{remark}

\begin{proof}
Let $\E$ be stable and let $a\in\mathrm{Hom}(\E,\E)$ be non-zero. Suppose that $a$ is {\em not\/} an isomorphism. Denote by $\F$ and $\G$ the kernel and image of $a$, respectively. Then both $\F$ and $\G$ are proper subsheaves of $\E$, so $p_{H,\F}(n)<p_{H,\E}(n)$ and $p_{H,\G}(n)<p_{H,\E}(n)$ for $n\gg 0$, by the stability of $\E$. Now from the short exact sequence
$$0\longrightarrow \F\longrightarrow\E\longrightarrow\G\longrightarrow 0$$
we get
$$\chi(\E\otimes H^n)=\chi(\F\otimes H^n)+\chi(\G\otimes H^n),$$
and therefore
$$p_{H,\E}(n)=\frac{\chi(\E\otimes H^n)}{\mathrm{rank}\,\E}=\frac{\chi(\F\otimes H^n)+\chi(\G\otimes H^n)}{\mathrm{rank}\,\F+\mathrm{rank}\,\G}=\frac{p_{H,\F}(n)\mathrm{rk}\,\F+p_{H,\G}(n)\mathrm{rk}\,\G}{\mathrm{rank}\,\F+\mathrm{rank}\,\G}<p_{H,\E}(n).$$
This contradiction shows that the non-zero endomorphism $a$ must be an isomorphism. Next, let $\lambda$ be an eigenvalue of the induced isomorphism $\bar{a}:\H^0(\E\otimes H^n)\rightarrow\H^0(\E\otimes H^n)$, for some $n\gg 0$, and let $b:=a-\lambda\mathrm{Id}\in\mathrm{Hom}(\E,\E)$. Since the induced map $\bar{b}:\H^0(\E\otimes H^n)\rightarrow\H^0(\E\otimes H^n)$ is not an isomorphism, $b$ itself cannot be an isomorphism. This means that $b$ must be zero. Thus $a=\lambda\mathrm{Id}$ and $\E$ must be simple, proving 1.

By Riemann-Roch, we have
$$\chi(\E\otimes H^n)=\int_X\mathrm{ch}(\E)\exp(nH)\mathrm{Td}_X=\mathrm{rank}\,\E\frac{H^2}{2!}n^2+c_1(\E).Hn+\mbox{const.}$$
Therefore
$$p_{H,\E}(n)=\frac{\chi(\E\otimes H^n)}{\mathrm{rank}\,\E}=\frac{H^2}{2!}n^2+\frac{c_1(\E).H}{\mathrm{rank}\,\E}n+\mbox{const}=\frac{H^2}{2!}n^2+\mu(\E)n+\mbox{const.}$$
Statement 2 now follows by elementary algebra. For instance, suppose $\E$ is $\mu$-stable. Then $\mu(\F)<\mu(\E)$ for all coherent subsheaves $\F\subset\E$ with $0<\mathrm{rank}\,\F<\mathrm{rank}\,\E$. This implies that
$$p_{H,\F}(n)=\frac{H^2}{2!}n^2+\mu(\F)n+\mbox{const}<\frac{H^2}{2!}n^2+\mu(\E)n+\mbox{const}=p_{H,\E}(n)$$
for $n\gg 0$, and thus $\E$ is stable.

Statement 3 also follows by elementary algebra. Suppose $\E$ is $\mu$-semistable. Then $\mu(\F)\leq \mu(E)$ for all coherent subsheaves $\F\subset\E$ with $0<\mathrm{rank}\,\F<\mathrm{rank}\,\E$. If $\E$ is {\em not\/} $\mu$-stable then we must have equality for some $\F$, i.e.,
$$\mu(\F)=\frac{c_1(\F).H}{\mathrm{rank}\,\F}=\frac{c_1(\E).H}{\mathrm{rank}\,\E}=\mu(\E).$$
Because $H$ is $v$-general, this can only happen if $c_1(\F)/\mathrm{rank}\,\F=c_1(\E)/\mathrm{rank}\,\E$, or equivalently, $c_1(\F)\mathrm{rank}\,\E=c_1(\E)\mathrm{rank}\,\F$. Thus $\mathrm{rank}\,\E$ divides $c_1(\E)\mathrm{rank}\,\F$, where $\mathrm{rank}\,\F<\mathrm{rank}\,\E$, contradicting the fact that $\mathrm{rank}\,\E$ and the divisibility of $c_1(\E)$ are coprime.

Finally, statement 4 is proved in a similar manner. If $\E$ is semistable but {\em not\/} stable, then there exists a proper subsheaf $\F\subset \E$ such that $p_{H,\F}(n)= p_{H,\E}(n)$ for $n\gg 0$. Because $H$ is $v$-general, we must have $v(\F)\mathrm{rank}\,\E =v(\E)\mathrm{rank}\,\F$. If $\mathrm{rank}\,\F<\mathrm{rank}\,\E$ then some prime factor of $\mathrm{rank}\,\E$ divides $v(\E)$, contradicting the primitivity of $v(\E)$. On the other hand, if $\mathrm{rank}\,\F=\mathrm{rank}\,\E$ then the quotient sheaf $\E/\F$ is torsion. Moreover
$$\chi(\F\otimes H^n)=p_{H,\F}(n)\mathrm{rank}\,\F=p_{H,\E}(n)\mathrm{rank}\,\E=\chi(\E\otimes H^n),$$
so $\chi((\E/\F)\otimes H^n)=0$ for $n\gg 0$. This implies that $\E/\F=0$, contradicting the fact that $\F$ is a proper subsheaf of $\E$.
\end{proof}

Gieseker~\cite{gieseker77} constructed the moduli space $M_H(v)$ of stable sheaves $\E$ on $X$ with Mukai vector $v(\E)=v$, as a quasi-projective scheme. If $M_H(v)$ is not already compact, then one can compactify by adding semistable sheaves. Now a semistable sheaf $\E$ admits a Jordan-H{\"o}lder filtration, $0=\E_0\subset \E_1\subset\ldots\subset \E_k=\E$, whose graded factors $\E_{i+1}/\E_i$ are stable sheaves with the same normalized Hilbert polynomials as $\E$. These graded factors are uniquely determined by $\E$, up to ordering, and two semistable sheaves are said to be $S$-equivalent if they have the same graded factors. Then Gieseker also constructed the moduli space $M_H(v)^{ss}$ of $S$-equivalence classes of semistable sheaves $\E$ on $X$ with Mukai vector $v(\E)=v$, as a projective scheme that compactifies $M_H(v)$. (Note that a stable sheaf $\E$ is also semistable, and its $S$-equivalence class consists of just the sheaf $\E$ itself.) In particular, if every semistable sheaf is stable, then $M_H(v)$ will already be projective, and hence compact. For example, by statement 4 of Lemma~\ref{2}, this occurs when $v$ is primitive and $H$ is $v$-general.

Gieseker's construction applies to arbitrary smooth projective surfaces $(X,H)$, though for K3 surfaces we can invoke Mukai's results on the moduli space of simple sheaves. By statement 1 of Lemma~\ref{2}, the moduli space $M_H(v)$ of stable sheaves is an (open) subscheme of $\mathrm{Spl}(v)$. Therefore for K3 surfaces, $M_H(v)$ is smooth and it inherits a closed non-degenerate holomorphic two-form.

Notice that as the polarization $H$ varies, we generally get different open subschemes $M_H(v)\subset\mathrm{Spl}(v)$. For example, if $H$ and $H^{\prime}$ lie on opposite sides of a wall of the ample cone, then $M_H(v)$ and $M_{H^{\prime}}(v)$ may be birational but not isomorphic. What is happening here is that some locus $\Gamma$ of $M_H(v)$ is removed and replaced by a locus $\Gamma^{\prime}$ in $M_{H^{\prime}}(v)$, and for every point $p\in\Gamma\subset\mathrm{Spl}(v)$ there is a corresponding non-separated point $p^{\prime}\in\Gamma^{\prime}\subset\mathrm{Spl}(v)$ (recall that $\mathrm{Spl}(v)$ is non-Hausdorff!). We will say more about variations of the polarization in Section~5.

\begin{example}
Let $v=(1,0,1-n)$. A torsion free sheaf $\E$ with Mukai vector $v(\E)=(1,0,1-n)$ will have rank one, $c_1=0$, and $c_2=n$. It will be $\mu$-stable because there are no coherent subsheaves $\F\subset \E$ with $0<\mathrm{rank}\,\F<\mathrm{rank}\,\E=1$. It will also be stable with respect to all possible polarizations $H$ because if $\F\subset\E$ is a proper subsheaf then $\mathrm{rank}\,\F=\mathrm{rank}\,\E=1$, $\E/\F$ is torsion, and thus
$$p_{H,\E}(n)-p_{H,\F}(n)=\chi(\E\otimes H^n)-\chi(\F\otimes H^n)=\chi((\E/\F)\otimes H^n)>0$$
for $n\gg 0$. Now $\E^{\vee\vee}$ will be reflexive and therefore locally free, as $X$ is a surface; thus $\E^{\vee\vee}\cong\O_X$. The cokernel of the inclusion $\E\hookrightarrow\E^{\vee\vee}$ will be the structure sheaf $\O_Z$ of a zero-dimensional subscheme $Z\subset X$ of length $n$. Thus $\E$ is identified with the ideal sheaf $\I_Z$ of $Z$,
$$0\longrightarrow \E\cong\I_Z\longrightarrow \E^{\vee\vee}\cong\O_X\longrightarrow \O_Z\longrightarrow 0,$$
and the moduli space $M_H(1,0,1-n)$ is identified with the Hilbert scheme $\mathrm{Hilb}^nX$ of length $n$ zero-dimensional subschemes $Z\subset X$.
\end{example}

In low dimensions we can completely describe the moduli spaces.

\begin{theorem}[Mukai~\cite{mukai87}, Corollary~3.6]
\label{3}
If $\langle v,v\rangle=-2$ then $M_H(v)$ is empty or a single point.
\end{theorem}

%\begin{remark}
%It was later proved that $M_H(v)$ is always non-empty in this case.
%\end{remark}

\begin{proof}
Simple sheaves with $\langle v(\E),v(\E)\rangle=-2$ are known as rigid sheaves, because the space $\mathrm{Ext}^1(\E,\E)$ parametrizing their first order deformations is trivial. Mukai first shows that a rigid torsion free sheaf $\E$ must be locally free. Moreover, if $\E$ is stable and $\F$ is another semistable sheaf with $v(\F)=v(\E)$, then by Riemann-Roch
$$\chi(\E,\F)=-\langle v(\E),v(\F)\rangle=2.$$
So at least one of $\mathrm{Hom}(\E,\F)$ and $\mathrm{Hom}(\F,\E)\cong\mathrm{Ext}^2(\E,\F)^{\vee}$ must be non-trivial. But because $\E$ is stable, a homomorphism $\E\rightarrow\F$, respectively $\F\rightarrow\E$, must be an isomorphism.
\end{proof}

\begin{theorem}[Mukai~\cite{mukai87}, Theorem~1.4, Corollary~4.6]
\label{4}
Fix $v\in\H^{\mathrm{ev}}(X,\Z)$ with $\langle v,v\rangle=0$ and fix an ample divisor $H$ on $X$. If every semistable sheaf $\E$ with $v(\E)=v$ is stable, then $M_H(v)$ is empty or a K3 surface.
\end{theorem}

\begin{remark}
As we have already observed, if $v$ is primitive and $H$ is $v$-general then the hypothesis of the theorem is satisfied.
\end{remark}

\begin{proof}
Recall that the moduli space $\mathrm{Spl}(v)$ of simple sheaves with Mukai vector $v$ is smooth of dimension $\langle v,v\rangle+2=2$. As described above, $M_H(v)$ is an open subscheme of $\mathrm{Spl}(v)$. Assume $M_H(v)$ is non-empty. If every semistable sheaf $\E$ with $v(\E)=v$ is stable, then $M_H(v)$ contains a connected component $M$ that is compact. We claim that every semistable sheaf $\F$ with $v(\F)=v$ is isomorphic to a stable sheaf in the family parametrized by $M$. To see this, one takes a (quasi-)universal sheaf on $X\times M$, uses it to construct a Fourier-Mukai transform $\Phi:D^b(X)\rightarrow D^b(M)$, and then shows that $\Phi(\F)$ is the shift of a skyscraper sheaf on $M$. Applying the inverse Fourier-Mukai transform then proves the claim (see the proof of Proposition~4.4~\cite{mukai87} for details). It follows that $M_H(v)=M$ is irreducible. Moreover, $\mathrm{Spl}(v)$, and therefore $M_H(v)$, admits a nowhere vanishing holomorphic two-form, so $M_H(v)$ is an abelian or K3 surface. Finally, the cohomological Fourier-Mukai transform induces an isometry
$$\varphi_{\mathbb{Q}}:v^{\perp}/\mathbb{Q}v\longrightarrow \H^2(M_H(v),\mathbb{Q}),$$
where $v^{\perp}$ denotes the orthogonal complement of $\mathbb{Q}v$ in $\H^{\mathrm{ev}}(X,\Z)$. We conclude that $M_H(v)$ is a K3 surface.
\end{proof}

Mukai also proved non-emptyness of these moduli spaces under certain conditions, a result which was generalized by Yoshioka. Specifically, we say that a Mukai vector $v\in\H^{\mathrm{ev}}(X,\Z)$ with $v_2\in\mathrm{NS}(X)$ is positive if $v_0>0$, or if $v_0=0$ and $v_2$ is effective, or if $v_0=v_2=0$ and $v_4>0$. Then it follows from the general results in~\cite{yoshioka03i,yoshioka03ii} (and also Corollary~3.5 of~\cite{yoshioka09} for the rank zero case) that $M_H(v)$ is non-empty if $v$ is primitive, positive, and $\langle v,v\rangle\geq -2$.
%Then it follows from Corollary~3.5 of~\cite{yoshioka09} that $M_H(v)$ is non-empty if $v$ is primitive, positive, and $\langle v,v\rangle\geq -2$.
%Then it follows from the general results in~\cite{yoshioka03i,yoshioka03ii} that $M_H(v)$ is non-empty if $v$ is primitive, positive, and $\langle v,v\rangle\geq -2$. 

In higher dimensions we have the following result, the culmination of the work of a number of authors, including Mukai, G{\"o}ttsche and Huybrechts, O'Grady, and Yoshioka. First a definition.

\begin{definition}
A holomorphic symplectic manifold is a compact K{\"a}hler manifold $M$ that admits a non-degenerate holomorphic two-form $\sigma\in\H^0(M,\Omega^2_M)$. Moreover, we say $M$ is irreducible if it is simply connected and $\H^0(M,\Omega^2_M)$ is generated by $\sigma$.
\end{definition}

\begin{remark}
Non-degeneracy means that taking the interior product with $\sigma$ induces an isomorphism $TM\cong \Omega^1_M$; equivalently, $\sigma^n$ is nowhere vanishing and trivializes $K_M=\Omega^{2n}_M$, where $2n$ is the dimension of $M$ (which is necessarily even). In addition, $\sigma$ is $d$-closed because Hodge theory implies
$$\H^0(M,\Omega^2_M)\subset \H^2(M,\C).$$
\end{remark}

\begin{theorem}
\label{5}
Let $v\in\H^{\mathrm{ev}}(X,\Z)$ be positive, primitive, and let $H$ be $v$-general. Suppose that $2n:=\langle v,v\rangle+2$ is greater than $2$. Then $M_H(v)$ is an irreducible holomorphic symplectic manifold that is deformation equivalent to the Hilbert scheme $\mathrm{Hilb}^nX$ of $n$ points on a K3 surface. In addition, the weight-two Hodge structure $\H^2(M_H(v),\Z)$ is isomorphic to $v^{\perp}$, and this isomorphism takes the Beauville-Bogomolov quadratic form on $\H^2(M_H(v),\Z)$ to the restriction of the Mukai pairing on $v^{\perp}$.
\end{theorem}

\begin{remark}
Here $v^{\perp}$ denotes the orthogonal complement of $v$ in $\H^{\mathrm{ev}}(X,\Z)$, which is given the Hodge structure with $(1,1)$-part equal to
$$\H^0(X,\Z)\oplus\H^{1,1}(X,\Z)\oplus\H^4(X,\Z).$$
The Beauville-Bogomolov form~\cite{beauville83} of an irreducible holomorphic symplectic manifold $M$ is a natural quadratic form on $\H^2(M,\Z)$, given by taking
$$\tilde{q}(\alpha)=\frac{n}{2}\int_M(\sigma\bar{\sigma})^{n-1}\alpha^2+(1-n)\int_M\sigma^{n-1}\bar{\sigma}^n\alpha\cdot\int_M\sigma^n\bar{\sigma}^{n-1}\alpha$$
and rescaling by a rational constant (to make it primitive and integral).
\end{remark}

\begin{proof}
As describe above, Mukai~\cite{mukai84,mukai88} proved that the moduli space $\mathrm{Spl}(v)$ of simple sheaves is smooth of dimension $\langle v,v\rangle+2=2n$, and it admits a non-degenerate holomorphic two-form. Moreover, $M_H(v)$ is an open subscheme, that is also compact under the hypothesis that $v$ is primitive and $H$ is $v$-general. The relation between $M_H(v)$ and $\mathrm{Hilb}^nX$ was explored first in the rank two case, i.e., when $v_0=2$, by G{\"o}ttsche and Huybrechts~\cite{gh96}, and then for higher rank but with primitive $v_2\in\H^2(X,\Z)$ by O'Grady~\cite{ogrady97}. O'Grady's idea is to deform the underlying K3 surface $X$ to a certain elliptic K3 surface, inducing a corresponding deformation of $M_H(v)$. Then one can implement a backward induction, eventually reducing the rank of the sheaves to one, i.e., at each step, one identifies $M_H(v)$ with a moduli space of stable sheaves whose rank is one less than for $v$ (the identification is a birational map, but Huybrechts~\cite{huybrechts99} later proved that birational holomorphic symplectic manifolds are also deformation equivalent). Rank one torsion free sheaves on a K3 surface look like line bundles tensored with ideal sheaves of zero-dimensional subschemes, and thus we can identify the moduli space with $\mathrm{Hilb}^nX$. Finally, Yoshioka removed the hypothesis that $v_2$ be primitive (Theorem~8.1 of~\cite{yoshioka01}) and treated the rank zero case, i.e., pure dimension one sheaves~\cite{yoshioka03ii}. Yoshioka's approach differs from O'Grady's: he proves that stability is preserved under Fourier-Mukai transforms, under certain conditions, and uses this to construct isomorphisms between different moduli spaces. He also proved a generalization of the theorem for non-primitive $v$ and studied non-general $H$ in~\cite{yoshioka03i}.

Mukai~\cite{mukai88} constructed a canonical map
$$\theta_v:v^{\perp}\longrightarrow\H^2(M_H(v),\Z)$$
using the Chern class of a quasi-universal sheaf on $X\times M_H(v)$. By carefully tracking this map under the birational identifications above, O'Grady~\cite{ogrady97} proved that it is an isomorphism of Hodge structures, and an isometry of lattices taking the Mukai pairing on $v^{\perp}$ to the Beauville-Bogomolov form on $\H^2(M_H(v),\Z)$, provided that $v_2$ is primitive (a technical assumption that was later removed by Yoshioka~\cite{yoshioka01}).
\end{proof}

\subsection{Singular moduli spaces}

Theorem~\ref{5} gives a complete description of $M_H(v)$ for primitive $v$ and $v$-general $H$. For both non-primitive $v$ and non-general $H$, $M_H(v)$ will be non-compact, and one must add ($S$-equivalence classes of) strictly semistable sheaves to compactify, $M_H(v)\subset M_H(v)^{ss}$. We will discuss variations of the polarization $H$ in Section 5. For non-primitive $v$, the strictly semistable locus is the singular locus of $M_H(v)^{ss}$; only in some special cases do these singularities admit a symplectic desingularization.

\begin{theorem}[O'Grady~\cite{ogrady99}]
For the Mukai vector $v=(2,0,-2)\in\H^{\mathrm{ev}}(X,\Z)$, the moduli space $M_H(2,0,-2)^{ss}$ admits a symplectic desingularization. The resulting smooth irreducible symplectic variety of dimension ten is not deformation equivalent to $\mathrm{Hilb}^5X$; indeed, it gives a new deformation class of irreducible holomorphic symplectic manifolds.
\end{theorem}

If $W$ and $Z\subset X$ are zero-dimensional subschemes of length two, then $\I_W\oplus \I_Z$ is a strictly semistable sheaf with Mukai vector $(2,0,-2)$; in fact, all strictly semistable sheaves in $M_H(2,0,-2)^{ss}$ are of this form, so the singular locus is isomorphic to
$$\mathrm{Sym}^2(\mathrm{Hilb}^2X).$$
O'Grady desingularizes $M_H(2,0,-2)^{ss}$ by a sequence of blowups, followed by a blowdown to ensure that the induced holomorphic two-form on the resulting smooth space is non-degenerate.

In general, suppose that $v=mv_0\in\H^{\mathrm{ev}}(X,\Z)$ where $m\geq 2$ and $v_0$ is primitive. (Note that we have changed our notation: in this section $v_0$ denotes a Mukai vector, not the degree zero component of the Mukai vector $v$.) If $\langle v_0,v_0\rangle=-2$ then every semistable sheaf with Mukai vector $mv_0$ will be isomorphic to $\E^{\oplus m}$, where $\E$ is the unique stable sheaf with Mukai vector $v_0$ (see Theorem~\ref{3}); thus $M_H(v_0)^{ss}$ consists of a single point. If $\langle v_0,v_0\rangle=0$ then every semistable sheaf with Mukai vector $mv_0$ will be $S$-equivalent to a direct sum of $m$ stable sheaves with Mukai vectors $v_0$; thus $M_H(mv_0)^{ss}$ is isomorphic to $\mathrm{Sym}^mM_H(v_0)$, where $M_H(v_0)$ is a K3 surface by Theorem~\ref{4}. Assume now that $\langle v_0,v_0\rangle\geq 2$.

\begin{theorem}[Lehn and Sorger~\cite{ls06}]
If $\langle v_0,v_0\rangle=2$ and $H$ is $v$-general then the blowup of the reduced singular locus of $M_H(2v_0)^{ss}$ is a symplectic desingularization. (In particular, O'Grady's example is of this form.)
\end{theorem}

\begin{theorem}[Kaledin, Lehn, and Sorger~\cite{kls06}]
If $m\geq 2$ and $\langle v_0,v_0\rangle>2$ or $m>2$ and $\langle v_0,v_0\rangle \geq 2$ then $M_H(mv_0)$ has locally factorial singularities. Therefore, since the singularities occur in codimension $\geq 4$, they cannot be resolved symplectically.
\end{theorem}

\begin{remark}
As stated, the theorem is for sheaves of non-zero rank, though Kaledin et al.\ also proved it for rank zero under some additional hypotheses.
\end{remark}

\subsection{The Hitchin-Kobayashi correspondence}

We end this section with a discussion of one of the differential geometric aspects of stable bundles. Let $E$ be a holomorphic bundle equipped with a Hermitian metric $\langle\cdot,\cdot\rangle$. Recall that there exists a unique unitary connection $\nabla$ that is compatible with the holomorphic structure on $E$, i.e.,
$$d\langle s_1,s_2\rangle =\langle\nabla s_1,s_2\rangle + \langle s_1,\nabla s_2\rangle$$
for all smooth sections $s_1$ and $s_2$ of $E$, and $\nabla^{0,1}=\bar{\partial}_E$. Then $\nabla^2:=\nabla\circ\nabla$ will lie in $\Omega^{1,1}(\mathrm{End}E)$.

\begin{definition}
Let $\omega$ be a K{\"a}hler form on $X$, and let $E$ be a holomorphic bundle on $X$. Then $\langle\cdot,\cdot\rangle$ is a Hermitian-Einstein metric on $E$ if its corresponding compatible unitary connection $\nabla$ satisfies
$$\nabla^2\wedge\omega=\lambda\mathrm{Id}_E\omega^2\in\Omega^{2,2}(\mathrm{End}E)$$
where $\lambda$ is a constant (equivalently, the contraction of $\nabla^2$ with $\omega$ equals $\lambda\mathrm{Id}_E$).
\end{definition}

\begin{remark}
Taking the trace of both sides and integrating gives
$$\int_X \mathrm{tr}(\nabla^2)\wedge\omega=\lambda\mathrm{rank}\,E\int_X\omega^2.$$
By Chern-Weil theory, $\mathrm{tr}(\nabla^2)=-2\pi i c_1(E)$. Thus
$$\lambda=\frac{-\pi i}{\mathrm{vol}X}\mu(E),$$
where
$$\mu(E):=\frac{\int_X c_1(E)\wedge\omega}{\mathrm{rank}\,E}$$
is the slope of $E$ with respect to the K{\"a}hler form $\omega$ (if $\omega$ comes from an ample divisor $H$ then this is just the slope with respect to $H$).
\end{remark}

\begin{example}
The space of two forms $\Omega^2(M)$ on a Riemannian four-manifold decomposes into a direct sum $\Omega^2_+(M)\oplus\Omega^2_-(M)$ of self-dual and anti-self-dual two forms. For a compact K{\"a}hler surface $(X,\omega)$, the complexification of $\Omega^2_+(X)$ is spanned by $(2,0)$-forms, $(0,2)$-forms, and $\omega$, whereas the complexification of $\Omega^2_-(X)$ is the orthogonal complement of $\omega$ in $\Omega^{1,1}(X)$. Let $E$ be a holomorphic bundle on $X$ with $c_1(E)=0$, and equip $E$ with a Hermitian metric and corresponding compatible unitary connection $\nabla$. If $\nabla$ is Hermitian-Einstein then
$$\nabla^2\wedge\omega=0,$$
because $E$ has slope $\mu(E)=0$. Thus $\nabla^2\in\Omega^{1,1}(\mathrm{End}E)$ and it is orthogonal to $\omega$, implying that $\nabla^2$ lies in $\Omega^2_-(\mathrm{End}E)$. We call $\nabla$ an anti-self-dual connection (an example of a Yang-Mills instanton). Conversely, an anti-self-dual connection on a compact K{\"a}hler surface is Hermitian-Einstein.
\end{example}

The main result we wish to describe is the Hitchin-Kobayashi correspondence. It was proved first for Riemann surfaces by Narasimhan and Seshadri~\cite{ns65}, then for projective surfaces by Donaldson~\cite{donaldson85}, and then for compact K{\"a}hler manifolds in arbitrary dimension by Uhlenbeck and Yau~\cite{uy86} (and independently, for smooth projective varieties in arbitrary dimension by Donaldson~\cite{donaldson87}).

\begin{theorem}
Let $E$ be a holomorphic bundle $E$ on a compact K{\"a}hler manifold $(X,\omega)$. Then $E$ admits a Hermitian-Einstein metric if and only if it is $\mu$-polystable with respect to $\omega$, i.e., it is a direct sum of $\mu$-stable bundles.
\end{theorem}

As an application, we consider tensor products of bundles.

\begin{proposition}
If $E$ and $F$ are $\mu$-stable bundles then $E\otimes F$ is $\mu$-polystable.
\end{proposition}

\begin{proof}
Equip $E$ and $F$ with Hermitian-Einstein metrics, with corresponding connections $\nabla_E$ and $\nabla_F$. The induced Hermitian metric on $E\otimes F$, with corresponding connection $\nabla_E\otimes\mathrm{Id}_F+\mathrm{Id}_E\otimes\nabla_F$, is easily seen to be Hermitian-Einstein too. Therefore $E\otimes F$ is $\mu$-polystable.
\end{proof}

\section{The Chow ring of a K3 surface}

The Chow ring $CH(X)$ of a K3 surface is the ring of algebraic cycles on $X$ modulo rational equivalence; the ring structure is given by intersection. The Chow ring is graded by codimension,
$$CH(X)=CH^0(X)\oplus CH^1(X)\oplus CH^2(X),$$
though in this article we will adopt the common alternative notation
$$CH(X)=CH_2(X)\oplus CH_1(X)\oplus CH_0(X).$$
Now $CH_2(X)=\Z[X]$ for all irreducible surfaces, and $CH_1(X)\cong\mathrm{Pic}(X)$ is a lattice for K3 surfaces. On the other hand, the group $CH_0(X)$ of $0$-cycles up to rational equivalence is infinite-dimensional and much more complicated; simple and stable sheaves, and their moduli, play a r{\^o}le in understanding it.

\begin{theorem}[Beauville-Voisin~\cite{bv04}]
Let $c_X$ be the class of a point lying on a rational curve in the K3 surface $X$. Then
\begin{enumerate}
\item $c_X$ is independent of the choice of point and rational curve,
\item if $C_1$ and $C_2$ are two curves in $X$, then the class of $C_1\cap C_2$ in $CH_0(X)$ is a multiple of $c_X$, and
\item the second Chern class $c_2(X)\in CH_0(X)$ is equal to $24c_X$.
\end{enumerate}
\end{theorem}

\begin{definition}
The Beauville-Voisin ring is defined to be
$$R(X):=CH_2(X)\oplus CH_1(X)\oplus\Z c_X=\Z[X]\oplus \mathrm{Pic}(X)\oplus\Z c_X.$$
\end{definition}

\begin{remark}
By statement 2 above, $R(X)$ is a subring of $CH(X)$. Moreover, the Mukai vector
$$v(L)=\mathrm{ch}(L)\mathrm{Td}^{\frac{1}{2}}_X=\left(1,c_1(L),\frac{c_1(L)^2}{2}+\frac{c_2(X)}{24}\right)$$
of a line bundle $L$ on $X$, regarded as an element of the Chow ring, must lie in $R(X)$.
\end{remark}

\begin{definition}
A spherical object on $X$ is a bounded complex of sheaves, $\E^{\bullet}\in D^b(X)$, such that
$$\mathrm{Ext}^k(\E^{\bullet},\E^{\bullet})\cong\left\{\begin{array}{lr} {\C} & \mbox{if }k=0\mbox{ or }2, \\
 0 & \mbox{otherwise.} \end{array} \right.$$
\end{definition}

Line bundles are spherical objects, as are rigid sheaves. Spherical objects are the generalization of rigid sheaves to objects of the derived category.

\begin{theorem}[Huybrechts~\cite{huybrechts10}]
\label{12}
Let $X$ be a projective K3 surface with Picard number $\rho(X)\geq 2$.
\begin{enumerate}
\item If $\E^{\bullet}\in D^b(X)$ is a spherical object, then $v(\E^{\bullet})\in R(X)$.
\item Derived equivalences preserve $R(X)$, i.e., if $\Phi: D^b(X)\rightarrow D^b(X^{\prime})$ is an equivalence of triangulated categories then $\Phi^{CH}(R(X))=R(X^{\prime})$, where $\Phi^{CH}: CH(X)\rightarrow CH(X^{\prime})$ is the induced isomorphism of Chow groups.
\end{enumerate}
\end{theorem}

\begin{remark}
Statement 1 also holds when $\rho(X)=1$. Huybrechts proved this under the additional hypothesis that $v(\E^{\bullet})=(r,kH,s)$, where $H$ generates $\mathrm{Pic}X$ and $k\equiv \pm 1$ (mod $r$). Without this hypothesis, we can argue as follows: If $r=0$ then
$$2=\chi(\E^{\bullet},\E^{\bullet})=-\langle v(\E^{\bullet}),v(\E^{\bullet})\rangle=-k^2H^2\leq 0,$$
a contradiction; so without loss of generality we can assume $r>0$. As explained in the proof of Corollary~3.3 of~\cite{huybrechts10}, there exists a spherical locally free sheaf $E$ such that $v(E)=v(\E^{\bullet})$ in cohomology. Corollary~2.6 of~\cite{huybrechts10} shows that in fact $v(E)=v(\E^{\bullet})$ in the Chow ring. Finally, rigid sheaves are simple, so the theorem of Voisin stated below shows that $c_2(E)\in S_0(X)$. This means that $v(E)\in R(X)$.
\end{remark}

O'Grady~\cite{ogrady13} introduced a filtration on $CH_0(X)$,
$$S_0(X):=\Z c_X\subset S_1(X)\subset S_2(X)\subset \ldots \subset S_g(X)\subset \ldots \subset CH_0(X),$$
by defining $S_g(X)$ to be the set of classes $[Z]+ac_X$, where $Z=p_1+\ldots+p_g$ is an effective $0$-cycle of length $g$ and $a\in\Z$. Equivalently, $S_g(X)$ is the set of elements of the form $\iota_*[Z]$, where $Z$ is a $0$-cycle on a genus $g$ curve $C$ and $\iota:C\rightarrow X$ is any non-constant morphism (not necessarily an inclusion). Rigid sheaves belong to zero-dimensional moduli spaces; therefore the following conjecture is a generalization of Theorem~\ref{12} of Huybrechts.

\begin{conjecture}[O'Grady~\cite{ogrady13}]
Suppose that $M_H(v)$ has dimension $2d$. If $\E$ is a stable sheaf in $M_H(v)$, or a semistable sheaf whose $S$-equivalence class $[\E]$ lies in $M_H(v)^{ss}$, then $c_2(\E)$ lies in $S_d(X)$. Moreover, as we vary $\E$ in $M_H(v)$, and $[\E]$ in $M_H(v)^{ss}$, the set of all second Chern classes $c_2(\E)$ is precisely the set of elements of $S_d(X)$ of the appropriate degree, i.e., 
$$\{c_2(\E)\;|\; [\E]\in M_H(v)^{ss}\}=\{[Z]\in S_d(X)\;|\;\deg Z=c_2(v)\}.$$
\end{conjecture}

The conjecture is true in many cases.

\begin{theorem}[O'Grady~\cite{ogrady13}]
Suppose that $M_H(v)$ is non-empty and one of the following holds:
\begin{enumerate}
\item $v_2\in\H^2(X,\Z)$ is primitive and equal to $H$, and $v_4\geq 0$,
\item the Picard number $\rho(X)\geq 2$, the rank $v_0$ is coprime to the divisibility of $v_2\in\H^2(X,\Z)$, and $H$ is $v$-general,
\item the rank $v_0\leq 2$ and $H$ is $v$-general if $v_0=2$.
\end{enumerate}
Then the above conjecture is true for $M_H(v)$.
\end{theorem}

\begin{theorem}[Voisin~\cite{voisin15i}]
If $E$ is a simple bundle with Mukai vector $v(E)=v$, then $c_2(E)\in S_d(X)$, where $2d$ is the dimension of $M_H(v)$. Moreover,
$$\{c_2(E)\;|\; E\in\mathrm{Spl}(v)\}=\{[Z]\in S_d(X)\;|\;\deg Z=c_2(v)\}.$$
\end{theorem}

\begin{remark}
Voisin's result is for bundles that are simple, rather than stable or semistable. But there are no restrictions on $\rho(X)$ or $v$, so in the locally-free case, this generalizes both Huybrechts's and O'Grady's results.
\end{remark}

In addition, Voisin~\cite{voisin15i} characterized $S_d(X)$ in terms of dimensions of rational equivalence classes. Namely, for $k> d\geq 0$,
$$S_d(X)\cap\{[Z]\;|\;\deg Z=k\}$$
is precisely the set of $[Z]\in CH_0(X)$ such that the subset of $\mathrm{Sym}^kX$ parametrizing effective cycles rationally equivalent to $[Z]$ is non-empty and contains a component of dimension $\geq k-d$. 

In higher dimensions, a variation of O'Grady's definition produces a filtration on the Chow ring $CH_0(X^{[n]})$ of the Hilbert scheme of points on the K3 surface $X$. Voisin~\cite{voisin15ii} again characterized this filtration in terms of dimensions of rational equivalence classes. This led to a picture, still partly conjectural, of the structure of the Chow ring of an arbitrary hyperk{\"a}hler (i.e., holomorphic symplectic) variety.

\section{Ample and nef cones of holomorphic symplectic manifolds}

Let $X$ be a smooth projective variety, and denote by $\N_1(X,\Z)\subset \H_2(X,\Z)$ the group of curve classes modulo homological equivalence. The Mori cone of curves is the convex cone $\mathrm{NE}_1(X)\subset \N_1(X,\R)=\N_1(X,\Z)\otimes\R$ generated by effective classes. The importance of $\mathrm{NE}_1(X)$, or rather its closure $\overline{\mathrm{NE}}_1(X)$, is that it determines the ample and nef cones of $X$, as follows.

%Let $X$ be a smooth projective variety with N{\'e}ron-Severi group $\N^1(X,\Z)\subset \H^2(X,\Z)$. The effective divisors generate a cone $\mathrm{NE}^1(X)\subset \N^1(X,\R)$ whose closure $\overline{\mathrm{NE}}^1(X)$ is known as the pseudoeffective cone. Similarly, we have the group $\N_1(X,\Z)\subset \H_2(X,\Z)$ of curve classes modulo homological equivalence, the cone of curves $\mathrm{NE}_1(X)\subset \N_1(X,\R)$ generated by effective classes, and its closure the pseudoeffective cone of curves $\overline{\mathrm{NE}}_1(X)$. The importance of $\overline{\mathrm{NE}}_1(X)$ is due to the fact that it determines the ample and nef cones, as follows.

\begin{theorem}[Kleiman's criterion]
A divisor $D$ in $X$ is ample (respectively, nef) if and only if $D.C>0$ (respectively, $\geq 0$) for all curves $C\in \overline{\mathrm{NE}}_1(X)$.
\end{theorem}

An extremal ray $R$ of $\mathrm{NE}_1(X)$ is generated by the class of a rational curve $C\subset X$, and it yields a morphism $\varphi_R:X\rightarrow Z$ that contracts all rational curves equivalent to $C$. In this way, the cone of curves describes the birational geometry of $X$.

\begin{example}
Let $X$ be a projective K3 surface with polarization $H$. Then the cone of curves is given by
$$\mathrm{NE}_1(X)=\langle C\in\N_1(X,\Z)\;|\; C^2\geq -2,C.H>0\rangle,$$
where $\langle\phantom{xx}\rangle$ denotes the cone spanned by the given curve classes. In this case, $\mathrm{NE}_1(X)$ is already closed, so a divisor $D$ in $X$ is ample if $D.C>0$ for all curves $C\subset X$ with $C^2\geq -2$ and $C.H>0$. An extremal ray $R$ of the cone of curves will be generated by a $(-2)$-curve, and the morphism $\varphi_R:X\rightarrow Z$ will contract this curve, thereby producing a rational double point singularity in $Z$.
\end{example}

For higher-dimensional holomorphic symplectic manifolds, Hassett and Tschinkel introduced the following conjectural description of the cone of curves. Recall that we have the Beauville-Bogomolov integral quadratic form $q(\phantom{x},\phantom{x})$ on $\H^2(X,\Z)\cong\H^{2n-2}(X,\Z)$. We will denote the dual form on
$$\H^{2n-2}(X,\Z)^*\cong \H_2(X,\Z)$$
also by $q(\phantom{x},\phantom{x})$; note that it is $\Q$-valued since the Beauville-Bogomolov form is not necessarily unimodular.

\begin{thesis}[Hassett-Tschinkel~\cite{ht10}]
Let $X$ be an irreducible holomorphic symplectic manifold of dimension $2n$ with polarization $H$.
\begin{enumerate}
\item There exists a positive rational number $c_X$, depending only on the deformation class of $X$, such that
$$\mathrm{NE}_1(X)=\langle C\in\N_1(X,\Z)\;|\; q(C)\geq -c_X,C.H>0\rangle.$$
\item The extremal case $q(C)=-c_X$ arises as follows: if $\ell$ is a line in a Lagrangian $\P^n\subset X$ then $q(\ell)=-c_X$.
\end{enumerate}
\end{thesis}

\begin{remark}
When $X$ is a deformation of the Hilbert scheme of $n$ points on a K3 surface (in particular, when $X$ is a moduli space $M_H(v)$ of stable sheaves on a K3 surface with $v$ positive, primitive, and $H$ $v$-general), Hassett and Tschinkel~\cite{ht10} gave a conjectural value of $c_X=\frac{n+3}{2}$.
%and when $X$ is a deformation of the generalized Kummer variety of dimension $2n$ the conjectural value of $c_X$ is $\frac{n+1}{2}$~\cite{ht10}.
They also gave proposals for the birational maps resulting from various values of $q(R)<0$. They proved that an extremal ray $R$ corresponding to a divisorial contraction satisfies $-2\leq q(R)<0$ (this is Theorem~2.1 in~\cite{ht10}).
\end{remark}

\begin{example}
Let $X$ be the Hilbert scheme $\mathrm{Hilb}^nS$ of $n$ points on a K3 surface $S$. The exceptional divisor $E$, i.e., the locus of non-reduced subschemes, is $2$-divisible in $\H^2(X,\Z)$; write $E=2\delta$. Then
$$\H^2(X,\Z)\cong \H^2(S,\Z)\oplus\Z\delta$$
is an isomorphism of lattices that takes the Beauville-Bogomolov form $q$ on the left hand side to the direct sum of the intersection pairing on $\H^2(S,\Z)$ and $(\delta,\delta)=-2(n-1)$ on the right hand side. Let $C$ be a generic fibre of $E$ over the `big' diagonal in $\mathrm{Sym}^nS$. Then $C$ is a rational curve and $E|_C$ is isomorphic to $\O(-2)$. This implies that
$$C=-\delta^{\vee}\in\mathrm{NE}_1(X)\subset \H_2(X,\R)\cong\H^2(X,\R)^*.$$
Note that
$$q(C)=q(\delta^{\vee})=\frac{1}{q(\delta)}=\frac{-1}{2(n-1)}.$$
The corresponding contraction is of course the Hilbert-Chow morphism $\mathrm{Hilb}^nS\rightarrow\mathrm{Sym}^nS$.
\end{example}

\begin{example}
Now suppose that the K3 surface $S$ contains a rational $(-2)$-curve $C$. Then $\mathrm{Sym}^nC$ gives a Lagrangian $\P^n$ in $X=\mathrm{Hilb}^nS$. A line $\ell$ in this $\P^n$ is given by fixing $n-1$ points in $C$ and allowing the $n$th point to vary. To identify the class of $\ell$ in $\H_2(X,\R)\cong\H^2(X,\R)^*$ we intersect it with various divisors. Firstly
$$\ell. \delta=\frac{1}{2}\ell. E=n-1,$$
as each intersection of $\ell$ with $E$ has multiplicity $2$. Secondly, let $D$ be the divisor of subschemes whose support intersects $C$. Because $\ell$ is equivalent to a rational curve in $X$ given by one point on $C$ and $n-1$ fixed points {\em not\/} on $C$,
$$\ell. D=C^2=-2.$$
Since $q(D)=-2$, we conclude that
$$\ell=D+(n-1)\delta^{\vee}.$$
Finally, we can calculate
$$q(\ell)=q(D)+(n-1)^2q(\delta^{\vee})=-2+(n-1)^2\left(\frac{-1}{2(n-1)}\right)=-\frac{n+3}{2}.$$
The corresponding contraction collapses the Lagrangian $\P^n\subset X$ to a point.
\end{example}

\begin{remark}
The last example verifies the second part of Hassett and Tschinkel's conjecture in this case, but it does not prove it outright because a priori there could exist other Lagrangian $\P^n$s in different monodromy group orbits of $X$. Nonetheless, the fact that $q(\ell)=-\frac{n+3}{2}$ for a line $\ell$ in {\em any\/} Lagrangian $\P^n$ in a deformation $X$ of $\mathrm{Hilb}^nS$ was proved for $n=2$ by Hassett and Tschinkel~\cite{ht09}, for $n=3$ by Harvey, Hassett, and Tschinkel~\cite{hht12}, and for $n=4$ by Bakker and Jorza~\cite{bj14}.
\end{remark}

Returning to the first part of Hassett and Tschinkel's conjecture, they proved the following.

\begin{theorem}[Hassett-Tschinkel~\cite{ht09}]
Let $X$ be a deformation of the Hilbert scheme of two points on a K3 surface. Then a divisor $D$ in $X$ is ample if $D.C>0$ for all curves $C$ in
$$\langle C\in\N_1(X,\Z)\;|\; q(C)\geq -5/2,C.H>0\rangle.$$
In fact, it is enough to consider curves $C$ with $C.H>0$ and $q(C)=-5/2$, $-2$, $-1/2$, or $\geq 0$.
\end{theorem}

Thus Hassett and Tschinkel established sufficient conditions for $D$ to be ample in this case; they did not prove that these conditions are necessary. In fact, in general they are not: Proposition~10.3 and Remark~10.4 of Bayer and Macr{\`i}~\cite{bm14i} (and also independent work of Markman) shows that the cone of curves could be smaller than predicted, though the first instance of this happening is in dimension ten. Theorem~12.2 of Bayer and Macr{\`i}~\cite{bm14ii} gives the correct description of the Mori cone of curves when $X$ is a moduli space of sheaves on a K3 surface, and this was extended to the case when $X$ is an arbitrary deformation of $\mathrm{Hilb}^nS$ by Bayer, Hassett, and Tschinkel~\cite{bht15}. The precise statement is somewhat lengthy; in any case, we will say more about Bayer and Macr{\`i}'s methods in the next section. These methods also led to a verification of the bounds on the curves generating extremal rays.

\begin{theorem}[Bayer, Hassett, and Tschinkel~\cite{bht15}, Mongardi~\cite{mongardi15}]
Let $X$ be a deformation of $\mathrm{Hilb}^nS$, and let $R$ be an extremal ray of $\mathrm{NE}_1(X)$. Then $R$ contains an effective curve $C$ with $q(C)\geq -\frac{n+3}{2}$.
\end{theorem}

\begin{remark}
Bayer et al.\ proved this result in the projective case, whereas Mongardi's argument, obtained independently, works also in the non-projective case.
\end{remark}

\section{Moduli of stable complexes}

Fourier-Mukai transforms are equivalences
$$\Phi:D^b(X)\longrightarrow D^b(X^{\prime})$$
of triangulated categories~\cite{mukai81}. They have become an indispensible tool in studying moduli spaces because they induce birational maps of the form
\begin{eqnarray*}
M_{(X,H)}(v) & \dashrightarrow & M_{(X^{\prime},H^{\prime})}(v^{\prime}) \\
{\E} & \mapsto & \Phi({\E})^{\bullet}.
\end{eqnarray*}
Here $\Phi(\E)^{\bullet}$ is an element of the derived category $D^b(X^{\prime})$, i.e., a priori it is a complex of sheaves, but in nice situations it will be a sheaf (or shifted sheaf), and even a stable sheaf with respect to some polarization $H^{\prime}$ of $X^{\prime}$. If $\Phi(\E)^{\bullet}$ is a stable sheaf for {\em all} $\E$ in $M_{(X,H)}(v)$ then the above map will be an isomorpism. More generally, one would like to deal with moduli spaces of complexes of sheaves, and this is one of the motivations behind extending the notion of stability to complexes. Henceforth, we will drop the $\bullet$ and write a complex of sheaves $\E^{\bullet}$ simply as $\E$.

\begin{definition}[Bridgeland~\cite{bridgeland07}]
A stability condition $(Z,\mathcal{P})$ on $D^b(X)$ consists of a homomorphism $Z:K(X)\rightarrow\C$ from the Grothendieck $K$-group of $X$ and full additive subcategories $\mathcal{P}(\phi)\subset D^b(X)$ for each $\phi\in\R$, that satisfy:
\begin{enumerate}
\item $Z(\E)\in\C$ is a positive multiple of $\exp(i\pi\phi)$ for all $\E\in\mathcal{P}(\phi)$,
\item $\mathcal{P}(\phi+1)=\mathcal{P}(\phi)[1]$ for all $\phi\in\R$,
\item if $\phi_1>\phi_2$ then $\mathrm{Hom}_{D^b(X)}(\E_1,\E_2)=0$ for all $\E_i\in\mathcal{P}(\phi_i)$,
\item every $\E\in D^b(X)$ admits a generalized Harder-Narasimhan filtration
$$\begin{array}{lcclcccclccc}
0={\E}_0 & \longrightarrow & & {\E}_1 & \longrightarrow & & {\E}_2 & \longrightarrow\cdots\longrightarrow & {\E}_{n-1} & \longrightarrow & & {\E}_n={\E}, \\
 \phantom{xxx}{[1]}\nwarrow & & \swarrow & [1]\nwarrow & & \swarrow & & & [1]\nwarrow & & \swarrow & \\
 & {\A}_1 & & & {\A}_2 & & & & & {\A}_n & & \\
\end{array}$$
where $\A_i\in\mathcal{P}(\phi_i)$ and $\phi_1>\phi_2>\cdots >\phi_n$.
\end{enumerate}
\end{definition}

\begin{remark}
The homomorphism $Z$ is known as the central charge; its argument $\pi\phi$ can be regarded as a generalization of the slope $\mu$ of a sheaf, or rather, of $\tan^{-1}\mu\in S^1$ (see the example below). The collection of subcategories $\mathcal{P}$ is known as a slicing of $D^b(X)$. The Harder-Narasimhan filtration of a sheaf has graded pieces of descending slope; the generalized Harder-Narasimhan filtration extends this to complexes of sheaves.
\end{remark}

\begin{theorem}[Bridgeland~\cite{bridgeland07}]
The space of stability conditions $\mathrm{Stab}(X)$ on $D^b(X)$ is a topological space that is locally homeomorphic to a complex vector space.
\end{theorem}

More precisely, for each connected component $\Sigma$ of $\mathrm{Stab}(X)$ there is a linear subspace $V(\Sigma)$ of $\mathrm{Hom}_{\Z}(K(X),\C)$, and the local homeomorphism $\Sigma\rightarrow V(\Sigma)$ is given by the forgetful map $(Z,\mathcal{P})\mapsto Z$, which takes a stability condition to its central charge.

\begin{example}
Suppose that $X$ is a K3 surface with polarization $H$. Naively one would like to define a stability condition with central charge $Z(\E)=\mathrm{rank}\,\E+ic_1(\E). H$, whose argument $\phi$ satisfies
$$\tan\pi\phi=\frac{c_1(\E). H}{\mathrm{rank}\,\E}=\mu_H(\E).$$
However, this homomorphism $Z:K(X)\rightarrow\C$ vanishes on sheaves with zero-dimensional support, whereas the central charge of a stability condition must be non-zero on all non-zero sheaves. Instead, Bridgeland~\cite{bridgeland08} found the following construction. Let $\beta+i\omega$ be a complexified K{\"a}hler class on $X$, i.e., $\beta$ and $\omega$ lie in $\mathrm{NS}(X)\otimes\R$ with $\omega$ in the ample cone. Define
$$Z(\E):=\langle\exp(\beta+i\omega),v(\E)\rangle =\frac{1}{2r}\left((c_1^2-2rs)+r^2\omega^2-(c_1-r\beta)^2\right)+i(c_1-r\beta).\omega,$$
where $\E$ has Mukai vector $v(\E)=(r,c_1,s)$. For generic $\beta$ and $\omega$ (more precisely, one requires $Z(\E)\not\in\R_{\leq 0}$ for all spherical sheaves $\E$ on $X$), Bridgeland described how to construct a stability condition with central charge $Z$. These stability conditions belong to a distinguished component $\mathrm{Stab}^{\dagger}(X)\subset\mathrm{Stab}(X)$ of the space of all stability conditions on $D^b(X)$.

Now define $\mathcal{N}(X):=\Z\oplus\mathrm{NS}(X)\oplus\Z$, with inclusion $\mathcal{N}(X)\otimes\C\subset \mathrm{Hom}_{\Z}(K(X),\C)$ given by the Mukai pairing. There is an open subset $\mathcal{P}(X)\subset\mathcal{N}(X)\otimes\C$ given by vectors whose real and imaginary parts span positive definite two-planes in $\mathcal{N}(X)\otimes\R$. Denote by $\mathcal{P}^+(X)\subset\mathcal{P}(X)$ the component containing $\exp (\beta+i\omega)$, and let $\mathcal{P}^+_0(X)\subset\mathcal{P}^+(X)$ be the complement of the union of all $\delta^{\perp}$, i.e., complex hyperplanes orthogonal to $\delta$, where $\delta$ ranges over all spherical classes in $\mathcal{N}(X)$ with $\langle\delta,\delta\rangle=-2$. Bridgeland proved that the forgetful map
\begin{eqnarray*}
\pi:\mathrm{Stab}^{\dagger}(X) & \longrightarrow & \mathcal{N}(X)\otimes{\C}\subset \mathrm{Hom}_{\Z}(K(X),{\C}) \\
(Z,\mathcal{P}) & \longmapsto & \exp(\beta+i\omega),
\end{eqnarray*}
is a covering map over $\mathcal{P}^+_0(X)$. The group of deck transformations is given by the subgroup of the group of autoequivalences of $D^b(X)$ that act trivially on cohomology and that preserve the component $\mathrm{Stab}^{\dagger}(X)\subset\mathrm{Stab}(X)$.
\end{example}

\begin{remark}
Although there is no Bridgeland stability condition on $D^b(X)$ with central charge $Z(\E)=\mathrm{rank}\,\E+ic_1(\E).H$, one can recover $H$-stability by letting $\sigma\in\mathrm{Stab}^{\dagger}(X)$ go to a `large volume limit'.
\end{remark}

For moduli spaces of sheaves, stability depends on the choice of polarization $H$ in the ample cone, which has a wall and chamber structure (dependent on $v$). As we vary $H$ inside a chamber, the moduli space does not change. But when $H$ hits a wall, some stable sheaves can become strictly semistable; and when $H$ passes into a new chamber, they can become unstable. Thus we may find that $M_{H_+}(v)$ and $M_{H_-}(v)$ are birational for $H_+$ and $H_-$ in adjacent chambers of the ample cone, and an analysis of which sheaves become strictly semistable for $H$ on the wall will allow us to describe the birational transformation in more detail.

Similarly, the space of stability conditions $\mathrm{Stab}(X)$ admits a wall and chamber structure (dependent on $v$). In some situations, the above wall-crossing procedure allows us to {\em construct\/} moduli spaces $M_{\sigma}(v)$ of stable complexes on $X$, by\begin{enumerate}
\item starting with a stability condition $\sigma$ for which all stable complexes $\E$ with Mukai vector $v(\E)=v$ are actually sheaves (and thus $M_{\sigma}(v)$ is simply a moduli space of stable sheaves on $X$),
\item and then varying the stability condition $\sigma$ and keeping track of the birational modifications that occur as we cross walls.
\end{enumerate}
As we cross a wall, the locus of sheaves that become unstable will be replaced by a new locus of sheaves, or complexes of sheaves, that are stable for the new stability condition. Obviously it is important to understand the shapes and positions of walls in $\mathrm{Stab}(X)$, and this has been studied by several authors; for example, Maciocia~\cite{maciocia14} proved some results for general projective surfaces.

\begin{example}
Arcara and Bertram~\cite{ab13} considered the case $v=(0,H,H^2/2)$ with Bridgeland stability conditions $\sigma_t$ given by $\beta+i\omega=\frac{1}{2}H+itH$ where $t\in\mathbb{R}_{>0}$. For sufficiently large $t$ (in fact, for $t>\frac{1}{2}$) all $\sigma_t$-stable complexes $\E$ with Mukai vector $v(\E)=v$ are actually sheaves. These sheaves look like $\iota_*L$, where $L$ is a rank-one torsion free sheaf on $C\in |H|$, and $\iota:C\hookrightarrow X$ is the inclusion; thus for $t>\frac{1}{2}$, $M_{\sigma_t}(v)$ is just the compactified relative Jacobian of the family of curves in the linear system $|H|$. As $t$ is decreased, we run into walls at certain values $t_d$. Arcara and Bertram described the Mukai flops
$$M_{\sigma_{t_d+\epsilon}}(v)\dashrightarrow M_{\sigma_{t_d-\epsilon}}(v)$$
that occur when we cross these walls. The universal object (complex of sheaves) on $X\times M_{\sigma_{t_d-\epsilon}}(v)$ is obtained from the universal object on $X\times M_{\sigma_{t_d+\epsilon}}(v)$ by an elementary modification.
\end{example}

In the above example, even after varying $t$ and performing Mukai flops, a generic point of $M_{\sigma_t}(v)$ will still represent a sheaf on $X$, which is a rather special case. Indeed, this procedure does not yield a method of constructing moduli spaces $M_{\sigma}(v)$ of stable complexes for general $v$. A priori, these most general moduli spaces exist as Artin stacks of finite type over $\C$, as proved by Toda~\cite{toda08}. But stronger results are possible by applying Fourier-Mukai transforms, as observed by Minamide, Yanagide, and Yoshioka~\cite{myy14} in the Picard rank $\rho(X)=1$ case and further developed in Section~7 of Bayer and Macr{\`i}~\cite{bm14i} for $\rho(X)>1$. Specifically, they show that for a generic stability condition $\sigma\in\mathrm{Stab}^{\dagger}(X)$ there is a Fourier-Mukai transform $\Phi:D^b(X)\rightarrow D^b(X^{\prime})$ from $X$ to another K3 surface $X^{\prime}$ that takes $\sigma$-stable complexes on $X$ with Mukai vectors $v$ to $\Phi(\sigma)$-semistable {\em sheaves\/} on $X^{\prime}$. Here $\Phi(\sigma)$ is the stability condition on $D^b(X^{\prime})$ coming from $\sigma$ and the homeomorphism $\mathrm{Stab}^{\dagger}(X)\rightarrow\mathrm{Stab}^{\dagger}(X^{\prime})$ induced by $\Phi$. Thus $M_{\sigma}(X)$ can be identified with a moduli space of semistable sheaves on $X^{\prime}$. (More precisely, there may be a non-trivial gerbe $\alpha\in\mathrm{Br}(X^{\prime})$, and the sheaves on $X^{\prime}$ will be $\alpha$-twisted.)

An important step in studying the birational geometry of these moduli spaces is the description of ample and nef line bundles. Bayer and Macr{\`i} gave a natural construction of a `polarization' on any family of semistable complexes of sheaves admitting a universal object.

\begin{theorem}[Bayer and Macr{\`i}~\cite{bm14i}]
Let $\sigma=(Z,\mathcal{P})$ be a Bridgeland stability condition on $X$. Let $S$ be a family of $\sigma$-semistable objects in $D^b(X)$ with Mukai vector $v$, and with a universal family $\E\in D^b(S\times X)$. Then there exists a natural divisor class $\ell_{\sigma}$ on $S$ that is nef, and moreover $\ell_{\sigma}.C=0$ for a curve $C\subset S$ if and only if $\E_t$ is $S$-equivalent to $\E_{t^{\prime}}$ for all $t$ and $t^{\prime}$ in $C$.
\end{theorem}

\begin{proof}
We just give the definition of the divisor class $\ell_{\sigma}$ and refer to~\cite{bm14i} for the proof of its properties. A divisor class is uniquely determined by its values on curves. Given a (projective, reduced, and irreducible) curve $C\subset S$, we can take its structure sheaf $\O_C$ and apply the integral transform $\Phi:D^b(S)\rightarrow D^b(X)$ coming from $\E$ to get $\Phi(\O_C)\in D^b(X)$. Then we define
$$\ell_{\sigma}.C:=\mathcal{I}m\left(\frac{Z(\Phi(\O_C))}{Z(v)}\right).$$
Of course this definition depends implicitly on the universal family $\E\in D^b(S\times X)$, but one can show that changing $\E$ by the pullback of a line bundle on $S$ does not change $\ell_{\sigma}.C$, and so we supress $\E$ from the notation.
\end{proof}

If the stability condition $\sigma$ lies in the interior of a chamber $\mathcal{C}$ of $\mathrm{Stab}(X)$ then every $\sigma$-semistable complex is actually $\sigma$-stable. This implies that if $M_{\sigma}(v)$ is the moduli space of $\sigma$-stable complexes on $X$ with Mukai vector $v$, then $\ell_{\sigma}$ will be ample on $M_{\sigma}(v)$ (in the Picard rank $\rho(X)=1$ case, the ampleness of $\ell_{\sigma}$ also appears as Corollary~5.17 of Minamide et al.~\cite{myy14}). The assignment $\sigma\mapsto\ell_{\sigma}$ then gives a map from $\mathcal{C}$ to the ample cone $\mathrm{Amp}(M_{\sigma}(v))$, and from the closure $\overline{\mathcal{C}}$ to the nef cone $\mathrm{Nef}(M_{\sigma}(v))$. 

Bayer and Macr{\`i}~\cite{bm14i} then studied the following situation. Suppose that $\sigma_0$ is a generic stability condition on the boundary of the closure $\overline{\mathcal{C}}$, i.e., on a wall of $\mathrm{Stab}(X)$, and suppose that the locus of strictly $\sigma_0$-semistable complexes in $M_{\sigma}(v)$ is codimension at least two. Then $\ell_{\sigma_0}$ is big and nef and induces a birational contraction $M_{\sigma}(v)\rightarrow Y$. If $\sigma_0$ lies on a wall separating two chambers $\mathcal{C}_+$ and $\mathcal{C}_-$ then we obtain two birational contractions as above. In fact, the resulting spaces $Y_+$ and $Y_-$ are isomorphic and we obtain a flop
$$\begin{array}{ccccc}
 M_{\sigma_+}(v) & & \dashrightarrow & & M_{\sigma_-}(v) \\
 & \searrow & & \swarrow & \\
 & & Y_+=Y_- & , &
\end{array}$$
where $\sigma_+\in\mathcal{C}_+$ and $\sigma_-\in\mathcal{C}_-$ are stability conditions near $\sigma_0$ but on opposite sides of the wall. Thus $M_{\sigma_-}(v)$ is a different birational model of $M_{\sigma_+}(v)$. Moreover, the maps
$$\overline{\mathcal{C}}_+\longrightarrow\mathrm{Nef}(M_{\sigma_+}(v))\qquad\mbox{and}\qquad\overline{\mathcal{C}}_-\longrightarrow\mathrm{Nef}(M_{\sigma_-}(v))$$
can be `glued' along the wall, and after extending this process to all chambers, we arrive at a map from the space $\mathrm{Stab}^{\dagger}(X)$ of stability conditions to the cone of movable divisors on $M_{\sigma_+}(v)$.

The main difficulty with this argument is that there also exist {\em totally semistable walls\/} that arise from spherical objects in $D^b(X)$. On these walls, every complex in $M_{\sigma_0}(v)$ will be strictly semistable, and the corresponding birational transformation could be a flop, a divisorial contraction, or even an isomorphism. The different possibilities are classified in Theorem~5.7 of Bayer and Macr{\`i}~\cite{bm14ii}.

These ideas have numerous applications to the geometry of moduli spaces of sheaves on K3 surfaces. For example, they can be used to determine the nef cone of certain Hilbert schemes of points on K3 surfaces.

\begin{theorem}[Bayer and Macr{\`i}~\cite{bm14i}, Proposition~10.3]
Let $X$ be a K3 surface with polarization $H$ of degree $H^2=2d$, and assume that the N{\'e}ron-Severi group is generated over $\Z$ by $H$. For $n\geq\frac{d+3}{2}$, the nef cone of $\mathrm{Hilb}^nX$ is generated by $H$ and $H-\frac{2d}{d+n}\delta$. (Here $2\delta$ denotes the locus of non-reduced subschemes, and we regard $H$ as a divisor on $\mathrm{Hilb}^nX$ via the isomorphism
$$\H^2(\mathrm{Hilb}^nX,\Z)\cong\H^2(X,\Z)\oplus\Z\delta,$$
as in Section~4.)
\end{theorem}
%Check Yoshioka's paper. Maybe this result was already known, but it is still an application of Bayer and Macri's ideas.

\begin{remark}
This result is also a special case of Proposition~4.39 of Yoshioka~\cite{yoshioka01}, applied to the moduli space $M_H(1,0,1-n)\cong\mathrm{Hilb}^nX$. Of course this earlier proposition is proved without reference to moduli spaces of stable complexes.
\end{remark}

\begin{remark}
Bayer and Macr{\`i} also identify an extremal ray $R$, as a curve of $S$-equivalent $\sigma_0$-semistable complexes, that is contracted by the morphism corresponding to $H-\frac{2d}{d+n}\delta$. This ray satisfies
$$q(R)=-\frac{n+3}{2}+\frac{(d+1)(2n-d-3)}{2n-2}\geq -\frac{n+3}{2},$$
revealing that the Mori cone can be smaller than predicted by Hassett and Tschinkel~\cite{ht10}.
\end{remark}

Another application is to the existence of Lagrangian fibrations on Hilbert schemes of points on K3 surfaces, i.e., fibrations over $\P^n$ whose fibres are Lagrangian with respect to the holomorphic symplectic form. The Hyperk{\"a}hler SYZ Conjecture asserts that an irreducible holomorphic symplectic manifold admits a (rational) Lagrangian fibration if and only if it contains an isotropic divisor, i.e., a divisor $D$ such that $q(D)=0$ where $q$ is the Beauville-Bogomolov form (see Huybrechts~\cite{ghj02} or the author's article~\cite{sawon03}). For a divisor
$$D=C+k\delta\in\H^2(\mathrm{Hilb}^nX,\Z)\cong \H^2(X,\Z)\oplus\Z\delta,$$
this means that $q(D)=C^2-2k^2(n-1)$ must vanish. It was proved, independently by Markushevich and the author, that this leads to a Lagrangian fibration if $H=C$ generates the N{\'e}ron-Severi group of $X$.

\begin{theorem}[Markushevich~\cite{markushevich06}, Sawon~\cite{sawon07}]
Let $X$ be a K3 surface with N{\'e}ron-Severi group generated over $\Z$ by the polarization $H$. If $H^2=2k^2(n-1)$ for some integer $k$, then $\mathrm{Hilb}^nX$ admits a Lagrangian fibration.
\end{theorem}

\begin{proof}
We just give an outline of the argument. For $w=(k,H,k(n-1))\in\H^{\mathrm{ev}}(X,\Z)$, the moduli space $M_H(w)$ of stable sheaves on $X$ is a K3 surface, which we denote by $X^{\prime}$, and there exists a twisted Fourier-Mukai transform $\Phi:D^b(X)\rightarrow D^b(X^{\prime},\beta)$. The twist $\beta$ is necessary because $M_H(w)$ is not a fine moduli space. For $Z\in\mathrm{Hilb}^nX$, one can show that the transform of the ideal sheaf $\mathcal{I}_Z$ is a sheaf concentrated in a single degree. In fact, $\Phi(\mathcal{I}_Z)[1]$ is a torsion sheaf supported on a curve $C^{\prime}$ in $X^{\prime}$. Therefore $\Phi$ induces an isomorphism of $\mathrm{Hilb}^nX$ with a so-called Beauville-Mukai integrable system~\cite{beauville99} on $X^{\prime}$, i.e., the relative compactified Jacobian of a complete linear system $|C^{\prime}|$ of curves on $X^{\prime}$. The latter is obviously a Lagrangian fibration, with the map to $|C^{\prime}|\cong\P^n$ given by taking supports.
\end{proof}

Theorem~10.8 of Bayer and Macr{\`i}~\cite{bm14i} extends the above theorem to the case where $C=2H$ is twice the generator of the N{\'e}ron-Severi group. Their argument is roughly as follows. The hypothesis is that $4H^2=2k^2(n-1)$ for some odd integer $k$. One considers the moduli spaces $M_{\sigma_t}(v)$ of stable complexes where $v=(1,0,1-n)$ and $\sigma_t$ is the family of Bridgeland stability conditions on $D^b(X)$ given by $\beta+i\omega=-\frac{2}{k}H+itH$, where $t>0$. For $t\gg 0$ we have $M_{\sigma_t}(v)=\mathrm{Hilb}^nX$. There is also a Fourier-Mukai transform $\Phi:D^b(X)\rightarrow D^b(X^{\prime},\beta)$, as above, that induces an isomorphism of moduli spaces
$$M_{\sigma_t}(v)\cong M_{\Phi(\sigma_t)}(\Phi(v)).$$
For $t$ close to $0$, the right hand side is a moduli space of sheaves on $X^{\prime}$; indeed it is a Beauville-Mukai system, and therefore a Lagrangian fibration. Finally, one describes what happens as $t$ varies from very large to very small: there are finitely many wall-crossings, and each wall-crossing induces a birational modification of the moduli space. It follows that $\mathrm{Hilb}^nX$ is birational to a Lagrangian fibration; we say that it admits a {\em rational Lagrangian fibration\/}. Moreover, Bayer and Macr{\`i} proved that every minimal model for $\mathrm{Hilb}^nX$ arises as a moduli space of stable complexes in $D^b(X)$ for some stability condition.

%The difficulty with the above argument is that $\Phi(\mathcal{I}_Z)$ is a genuine complex of sheaves in this case. Some deformation of stability conditions and wall-crossings are required to get back to a moduli space of sheaves. Since the wall-crossings produce birational maps, the conclusion is that $\mathrm{Hilb}^nX$ admits a rational Lagrangian fibration, or equivalently, that some minimal model for $\mathrm{Hilb}^nX$ admits a Lagrangian fibration.
%(xiv) Sometimes when we cross a wall everything becomes semistable on the wall, then become complexes on the other side of the wall.

In their subsequent paper~\cite{bm14ii}, Bayer and Macr{\`i} broadly generalized these results so as to apply to any moduli space $M_H(v)$ of stable sheaves on a K3 surface, not just Hilbert schemes of points. They showed that
\begin{enumerate}
\item every minimal model of $M_H(v)$ can be interpreted as a moduli space of $\sigma$-stable complexes with Mukai vector $v$ for some Bridgeland stability condition $\sigma\in\mathrm{Stab}^{\dagger}(X)$,
\item the chamber decomposition of the movable cone of $M_H(v)$ can be determined from the wall and chamber structure of the special component $\mathrm{Stab}^{\dagger}(X)\subset\mathrm{Stab}(X)$ described above,
\item and the Hyperk{\"a}hler SYZ Conjecture holds for $M_H(v)$.
\end{enumerate}
Since we have already given an indication of the ideas and methods involved, we refer the reader to the original paper~\cite{bm14ii} for precise statements and more details.

\begin{flushleft}
Department of Mathematics\hfill sawon@email.unc.edu\\
University of North Carolina\hfill www.unc.edu/$\sim$sawon\\
Chapel Hill NC 27599-3250\\
USA\\
\end{flushleft}

\end{document}